\documentclass{amsart}
\usepackage{graphicx}
\usepackage{amsfonts,amsmath,amstext,amsthm,amssymb,ifpdf}

\newif\ifpdf
\ifx\pdfoutput\undefined
    \RequirePackage[hypertex]{hyperref}
  \pdffalse
\else
  \ifnum\pdfoutput=1
    \pdftrue
    \RequirePackage[pdftex, bookmarks=false]{hyperref}
  \else
      \RequirePackage[hypertex]{hyperref}
    \pdffalse
  \fi
\fi

\newtheorem{theorem}{Theorem}

\newtheorem{lemma}{Lemma}
\newtheorem*{definition}{Definition}

\def\d{{\rm d\,}}

\def\supp{{\rm supp\,}}

\def\ZZ{\ensuremath{\mathbb Z}}
\def\ZI{\ensuremath{\mathbb I}}
\def\ZN{\ensuremath{\mathbb N}}

\def\ZT{\ensuremath{\mathbb T}}

\def\ZR{\ensuremath{\mathbb R}}

\def\md#1#2\emd{\ifx0#1
\begin{equation*} #2 \end{equation*}\fi  
\ifx1#1\begin{equation}#2\end{equation}\fi   
\ifx2#1\begin{align*}#2\end{align*}\fi   
\ifx3#1\begin{align}#2\end{align}\fi    
\ifx4#1\begin{gather*}#2\end{gather*}\fi  
\ifx5#1\begin{gather}#2\end{gather}\fi   
\ifx6#1\begin{multline*}#2\end{multline*}\fi  
\ifx7#1\begin{multline}#2\end{multline}\fi  
\ifx8#1\begin{multline*}\begin{split}#2\end{split}\end{multline*}\fi
\ifx9#1\begin{multline}\begin{split}#2\end{split}\end{multline}\fi
}
\newcommand {\e }[1]{(\ref{#1})}
\newcommand {\lem }[1]{Lemma \ref{#1}}
\newcommand {\trm }[1]{Theorem \ref{#1}}

\begin{document}

\title{On the sweeping out property for convolution operators of discrete measures}%
\author{G. A. Karagulyan}%
\address{Institute of Mathematics of Armenian
National Academy of Sciences, Baghramian Ave.- 24b, 0019,
Yerevan, Armenia}%
\email{g.karagulyan@yahoo.com}%

\subjclass[2000]{42B25}%
\keywords{discrete measures, bounded entropy theorem, sweeping out property, Bellow problem}%

\begin{abstract}
Let $\mu_n$ be a sequence of discrete measures on the unit $\ZT=\ZR/\ZZ$ with $\mu_n(0)=0$, and $\mu_n((-\delta,\delta))\to 1$, as $n\to\infty$.  We prove that the sequence of convolution operators $(f\ast\mu_n)(x)$ is strong sweeping out, i.e. there exists a set $E\subset\ZT$ such that
\md0
\lim\sup_{n\to\infty}(\ZI_E\ast\mu_n)(x)= 1,\quad \lim\inf_{n\to\infty}(\ZI_E\ast\mu_n)(x)= 0,
\emd
almost everywhere on $\ZT$.
\end{abstract}
\maketitle
\begin{section}{Introduction}
We consider bounded discrete measures
\md0
\mu=\sum_k m_k\delta_{x_k},\quad \sum_km_k<\infty,
\emd
on the circle $\ZT=\ZR/\ZZ$, where $X=\{x_k\}$ is a finite or countable set in $\ZT $ and $\delta_{x_k}$ is Dirac measure at $x_k$. Denote
\md0
S_\mu f(x)=\int_\ZR f(x+t)d\mu(t).
\emd
Let $\mu_n$ be a sequence of discrete measures satisfying
\md1\label{1}
\mu_n(0)=0 ,\quad
\mu_n((-\delta,\delta))\to 1,\hbox{ as }n\to\infty ,
\emd
for any $0<\delta\le 1/2$. It is clear if $f\in L^1(\ZT)$ is continuous at $x\in\ZT$ then
\md1\label{4}
S_{\mu_n}f(x)\to f(x),
\emd
and the convergence is uniformly if $f\in C(\ZT)$. The almost everywhere convergence problem in the case of general $f\in L^1(\ZT)$ is not trivial. J.~Bourgain in \cite{Bour} proved
\begin{theorem}[J. Bourgain]\label{T1}
If $x_k\searrow 0$ as $k\to\infty $, and
\md0
\mu_n=\frac{1}{n}\sum_{k=1}^n \delta_{x_k},
\emd
then there exists a function $f\in L^\infty$, such that $S_{\mu_n}f(x)$ diverges on a set
of positive measure.
\end{theorem}
In fact,  this theorem gave a negative answer to a problem due to A.~Bellow \cite{Bel} and
the proof is based on a general theorem often referred as Bourgain's entropy principle.
Applying his principle Bourgain was able to deduce an analogous theorems for Riemann sums
\md0
\frac{1}{n}\sum_{k=0}^{n-1}f\bigg(x+\frac{k}{n}\bigg),
\emd
and for the operators
\md0
\frac{1}{n}\sum_{k=1}^nf(kx).
\emd
We note, that first theorem was earlier obtained by W.~Rudin \cite{Rud} by different technique, and the second  by J.~Marstrand in \cite{Mar}. S.~Kostyukovsky
and A. Olevskii in \cite{KoOl}, using the same entropy principle, extended \trm{T1} for general discrete sequences satisfying \e{1}.

We found a new geometric proof for \trm{T1}, as well as for the result from \cite{KoOl}. Moreover, the method allows to obtain a stronger divergence for the operators \e{4}. So in this paper we prove
\begin{theorem}\label{T2}
If discrete measures $\mu_n$ satisfy \e{1}, then
there exists a set $E\subset \ZT$, such that
\md1\label{3}
\limsup_{n\to\infty}S_{\mu_n}\ZI_E(x)=1,\quad \liminf_{n\to\infty}S_{\mu_n}\ZI_E(x)=0
\emd
almost everywhere on $\ZT$.
\end{theorem}
The relations \e{3} for sequences of operators is called strong sweeping out property.
These kind of operators are investigated by M. Akcoglu, A. Bellow, R. L. Jones, V. Losert,
K.Reinhold-Larsson, M. Wierdl \cite{Akc1} and by M. Akcoglu, M. D. Ha, R. L. Jones \cite{Akc2}. In \cite{Akc1} strong sweeping out property for Riemann sums operators is obtained. In \cite{Akc2} authors prove a general version of Bourgain's entropy principle, which allows to deduce sweeping out properties for some operators, but the principle is not applicable for the operators $S_{\mu_n}$.
The proof of \trm{T2} is based on \lem{L6}. It will be obtained from \lem{L6} simply applying a general result proved in \cite{Kar}.

\end{section}
\begin{section}{Proof of theorem}
Let
\md1\label{seq}
X=\{x_i:\,i=1,\ldots ,l\},\quad 0<x_1\le x_2\le\ldots \le x_l<1,
\emd
be an arbitrary sequence of reals. Suppose
\md0
Y=\{y_i:\,i=1,\ldots ,\nu\}, \quad y_1<y_2<\ldots <y_\nu=x_l
\emd
is a maximal independent (with respect to rational numbers) subset of $X$ containing $x_l$. Then we have
\md0
x_k=r^{(k)}_1y_1+\ldots+r^{(k)}_\nu y_\nu,\quad k=1,2,\ldots,l,
\emd
for some rational numbers $r^{(k)}_i$. Let $p$ be the least common multiple of the denominators of  $r^{(k)}_i$. Then we get
\md1\label{a7}
x_k=\frac{n^{(k)}_1y_1+n^{(k)}_2y_2+\ldots +n^{(k)}_\nu y_\nu}{p},
\emd
for some $n^{(k)}_i\in\ZZ$. Denote
\md1\label{a8}
\tau=\max_{i,k}|n^{(k)}_i|,
\emd
and
\md1\label{a2}
\begin{array}{rl}
A_m=&\bigg\{ y=\frac{n_1y_1+n_2y_2+\ldots +n_\nu y_\nu}{p} ;\, n_i\in\ZZ,\\
&|n_i|\le m\tau,\, i=1,2,\ldots ,\nu-1,\,|n_\nu|\le \nu m\tau+1\bigg\}.
\end{array}
\emd

\begin{lemma}\label{L1}If \e{seq} is an arbitrary sequence with $\nu\ge2$, then for any interval $I\subset (-1,1)$ with $|I|\le y_\nu/p$ we have
\md1\label{a16}
\#\big( A_m\cap I\big)\sim \gamma m^{\nu-1}|I|\hbox{ as } m\to\infty,
\emd
where $\gamma=(2\tau)^{\nu-1}p/y_\nu$ is a constant depended on $X$.
\end{lemma}
\begin{proof}
It is easy to observe that
\md2
 A_m \cap I =\bigg\{&y=\frac{n_1y_1+\ldots +n_\nu y_\nu}{p}:\\
&\,n_1\frac{y_1}{y_\nu}+\ldots +n_{\nu-1}\frac{y_{\nu-1}}{y_\nu} \in \frac{p}{y_\nu} \cdot I+\ZZ\cap [-(\nu m\tau+1),(\nu m\tau+1)],\\
&|n_i|\le m\tau,\, i=1,2,\ldots ,\nu-1\bigg\}.
\emd
On the other hand if $y\in A_m \cap I$, then, by \e{a2} we have
\md0
|n_1\frac{y_1}{y_\nu}+\ldots +n_{\nu-1}\frac{y_{\nu-1}}{y_\nu}|\le \nu m\tau.
\emd
Using also the relation $|I|\le y_\nu/p$, we conclude
\md9\label{a24}
 A_m \cap I =&\bigg\{y=\frac{n_1y_1+\ldots +n_\nu y_\nu}{p}:\\
&\,n_1\frac{y_1}{y_\nu}+\ldots +n_{\nu-1}\frac{y_{\nu-1}}{y_\nu} \in \frac{p}{y_\nu} \cdot  I+\ZZ,\\
&|n_i|\le m\tau,\, i=1,2,\ldots ,\nu-1\bigg\}.
\emd
Since $y_1,\ldots,y_\nu$ are independent, the number
\md0
\theta=y_{\nu-1}/y_\nu
\emd
is irrational. Denoting
\md1\label{a21}
E_m=\left\{n_1\frac{y_1}{y_\nu}+\ldots +n_{\nu-2}\frac{y_{\nu-2}}{y_\nu}:\,|n_i|\le m\tau,\, i=1,2,\ldots ,\nu-2\right\}
\emd
from \e{a24} we get
\md1\label{a23}
\frac{p}{y_\nu} \cdot\big( A_m\cap I\big)
=\big(\{n_{\nu-1}\theta:\,|n_{\nu-1}|\le m\tau\}+E_m\big)\cap\big(\frac{p}{y_\nu} \cdot  I+\ZZ\big).
\emd
It is well known that $n\theta+t$, $n=1,2,\ldots $ ($n=-1,-2,\ldots $), is a uniformly distributed sequence. This implies
\md1\label{a25}
\frac{\#\big(\{n_{\nu-1}\theta:\,|n_{\nu-1}|\le m\tau\}+t\big)\cap \big(\frac{p}{y_\nu} \cdot  I+\ZZ\big)}{2m\tau}\to \frac{p|I|}{y_\nu},\hbox { as } m\to \infty,
\emd
for any $t\in \ZR$ and the convergence is uniformly. Since $y_1,\ldots ,y_{\nu-1}$ are independent from \e{a21} we obtain
\md0
|E_m|=(2m\tau+1)^{\nu-2}.
\emd
Finally, using \e{a23} and \e{a25}, we get
\md0
\#\big( A_m\cap I\big)=\#\bigg(\frac{p}{y_\nu} \cdot\big( A_m\cap I\big)\bigg)\sim 2m\tau\frac{p|I|}{y_\nu}|E_m|\sim  (2m\tau)^{\nu-1}\frac{p|I|}{y_\nu}.
\emd
\end{proof}
\begin{lemma}\label{L2}
For any set \e{seq} we have
\md1\label{b1}
 A_m\cap (-x_l,0)+X\subset A_{m+1}\cap (-x_l,x_l),\, m=1,2,\ldots,
\emd
where $A_m$ is defined in \e{a2}.
\end{lemma}
\begin{proof}
Take an arbitrary point $x\in \in A_m\cap (-x_l,0)$. According to the definition of $y_1,\ldots,y_\nu$ we will have
\md0
x=\frac{n_1y_1+n_2y_2+\ldots +n_\nu y_\nu}{p},
\emd
Then suppose $x_k\in X$ has representation \e{a7}. Since $x\in (-x_l,0)$ and $0<x_k\le x_l$ we get
\md1\label{a11}
x+x_k\in (-x_l,x_l).
\emd
On the other hand
\md0
x+x_k=\frac{(n_1+n^{(k)}_1)y_1+(n_2+n^{(k)}_2)y_2+\ldots +(n_\nu+n^{(k)}_\nu) y_\nu}{p},
\emd
and by \e{a2} \e{a8} we have
\md9\label{a12}
&|n_i+n^{(k)}_i|\le m\tau+\tau=(m+1)\tau ,i=1,2,\ldots ,\nu-1,\\
&|n_\nu+n^{(k)}_\nu|\le \nu m\tau+1+\tau<\nu (m+1)\tau.
\emd
This means $x+x_k\in A_{m+1}$. Combining \e{a11} and \e{a12} we get \e{b1}.
\end{proof}
\begin{lemma}\label{L3}
For any numbers $\delta>0$, $0<\varepsilon<1/3$ and measure
\md1\label{e0}
\mu=\sum_{k=1}^lm_k\delta_{x_k},\, m_k>0,\,0<x_1<x_2<\ldots<x_l,
\emd
there exists a real number $\lambda$, with $0<\lambda\le \delta $, such that
\md7\label{e2}
S_\mu \ZI_{\{t:\,\{ t/\lambda\}>\varepsilon\}}(x)\\
=\int_\ZT\ZI_{\{t:\,\{ t/\lambda\}>\varepsilon\}}(x+t)d\mu(t)>(1-3\varepsilon) |\mu|,\hbox{ as } \{x/\lambda\}<\varepsilon.
\emd
\end{lemma}
\begin{proof}
Denote
\md1\label{e3}
E_t=\{\lambda>0:\, \{t/\lambda\}\in  (\varepsilon,1-\varepsilon)\},\quad t>0.
\emd
It is clear
\md0
E_t=\bigcup_{k=0}^\infty\left(\frac{t}{k+1-\varepsilon},\frac{t}{k+\varepsilon}\right).
\emd
Hence if
\md0
r=\min\left\{\frac{\varepsilon x_1}{2(1-\varepsilon)},\delta\right\}
\emd
and $t\ge x_1$, we obtain
\md9\label{e4}
|E_{t}\cap [0,r]|>&\sum_{k>t/r }\left(\frac{t}{k+\varepsilon}-\frac{t}{k+1-\varepsilon }\right)\\
&=\sum_{k>t/r}\left(\frac{(1-2\varepsilon)t}{(k+\varepsilon)(k+1-\varepsilon)}\right)> (1-2\varepsilon)t\sum_{k>t/r}\frac{1}{(k+1)^2}\\
&>\frac{(1-2\varepsilon)tr}{t+2r}>\frac{(1-2\varepsilon)x_1r}{x_1+2r}
\ge\frac{(1-2\varepsilon)x_1r}{x_1+\varepsilon x_1/(1-\varepsilon)}\\
&=(1-2\varepsilon)(1-\varepsilon)r>(1-3\varepsilon)r.
\emd
Thus, denoting
\md0
F=\{t>0:\, \{t \}\in (\varepsilon,1-\varepsilon)\},
\emd
by \e{e3} we have
\md0
E_t=\{\lambda>0:\, t\in  \lambda F\}
\emd
and therefore, using \e{e4}, we get
\md9\label{e5}
\int_0^rS_\mu \ZI_{\lambda F}(0)d\lambda&=\int_0^r\int_\ZT \ZI_{\lambda F}(t)d\mu(t)d\lambda\\
&=\int_\ZT\int_0^r \ZI_{\lambda F}(t)d\lambda d\mu(t)
=\int_\ZT|E_t\cap [0,r]|d\mu(t)\\
&=\sum_{i=1}^lm_i|E_{x_i}\cap [0,r]|\ge (1-3\varepsilon)r|\mu|.
\emd
This implies
\md1\label{e6}
S_\mu\ZI_{\lambda F}(0)>(1-3\varepsilon)|\mu|
\emd
for some $0<\lambda\le r\le \delta$. From \e{e6} it follows that
\md1\label{e7}
S_\mu\ZI_{{\lambda F}+x}(x)>(1-3\varepsilon)|\mu|,\quad x\in \ZR.
\emd
It is clear
\md1\label{e8}
\bigcup_{x:\,\{ x/\lambda\}<\varepsilon}({\lambda F}+x)=\{t:\, \{ t/\lambda\}>\varepsilon\}.
\emd
Thus, using \e{e7} and \e{e8}, for any $x$, $\{ x/\lambda\}<\varepsilon$, we obtain
\md0
S_\mu\ZI_{\{t:\,\{ t/\lambda\}>\varepsilon\}}(x)\ge S_\mu\ZI_{{\lambda F}+x}(x)>(1-3\varepsilon)|\mu|.
\emd
This implies \e{e2} and lemma is proved.
\end{proof}
\begin{lemma}\label{L4}
For any measure \e{e0}and number $0<\varepsilon<1/3$ there exist finite sets $E,G\subset (-x_l,x_l)$ such that
\md3
&E\cap G=\varnothing,\quad \#E>\frac{\varepsilon\#G}{4},\label{h0}\\
&S_\mu\ZI_G(x)>(1-3\varepsilon)|\mu|,\quad x\in E. \label{h2}
\emd
\end{lemma}
\begin{proof}
Denote
\md1\label{h5}
U_\lambda=\{t\in (-x_l,0):\, \{ t/\lambda\}<\varepsilon\},\quad V_\lambda=\{t\in (-x_l,x_l):\, \{t/\lambda\}>\varepsilon\}
\emd
It is clear $|U_\lambda|\to \varepsilon x_l$ and $|V_\lambda|\to 2(1-\varepsilon)x_l$  as $\lambda\to 0$. On the other hand, by \lem{L3}, for $\lambda$ small enough we have \e{e2}. So we can fix $\lambda$ satisfying \e{e2} and the conditions
\md1\label{h3}
0<\lambda <x_1,\quad |V_\lambda|<2x_l,\quad |U_\lambda|>\frac{\varepsilon x_l}{2}.
\emd
Denote
\md1\label{h7}
E_m=A_m\cap U_\lambda,\quad G_m=A_{m+1}\cap V_\lambda.
\emd
Since the sets $U_\lambda $ and $V_\lambda$ are finite union of intervals in $(-1,1)$, according to \lem{L1} we have
\md0
\#E_m\sim \gamma m^{\mu-1}|U_\lambda|,\quad \#G_m\sim \gamma m^{\mu-1}|V_\lambda|
\emd
as $m\to\infty$. Hence for an integer $m$ large enough, denoting
\md0
E=E_m,\quad G=G_m
\emd
and taking into account \e{h3} we will have
\md1\label{h6}
\#E>\frac{\varepsilon \#G}{4}.
\emd
Besides, since $U_\lambda\cap V_\lambda=\varnothing$ we have $E\cap G=\varnothing$ and so \e{h0}. To show \e{h2} we take an arbitrary $x\in E$. Because of \e{h5} and \e{h7}we will have
\md0
x\in A_m\cap (-x_l,0),\quad \{ x/\lambda\}<\varepsilon.
\emd
From \lem{L2} we get $x+X\in A_{m+1}\cap (-x_l,x_l)$. Thus we get
\md0
S_\mu\ZI_G(x)=S_\mu\ZI_{V_\lambda}(x)=S_\mu\ZI_{\{t:\, \{\lambda t\}>\varepsilon\}}(x)
\emd
and therefore, since we have $ \{ x/\lambda\}<\varepsilon$, from \lem{L3} we obtain \e{h2}.
\end{proof}
For an arbitrary nonempty finite set $A\subset \ZR \setminus\{0\}$ we define
\md0
\d(A)=\left\{
  \begin{array}{ll}
    \min\{|x-y|:\, x,y\in A,\, x\neq y\}, & \hbox{ if } \#A\ge 2,\\
    |x|, & \hbox{ if } A= \{x\}.
  \end{array}
\right.
\emd
\begin{lemma}\label{L5}
Let $A_k\subset \ZR \setminus \{0\}$, $k=1,2,\ldots $, be a sequence of nonempty finite sets such that and
\md1\label{f1}
\max A_{k+1}\le \frac{1}{4}\cdot \d( A_k), \quad k=1,2,\ldots .
\emd
Then the equality
\md1\label{f2}
x_1+x_2+\ldots +x_n=y_1+y_2+\ldots +y_n,\quad x_i,y_i\in A_i,\, i=1,2,\ldots ,n
\emd
implies $x_i=y_i$, $i=1,2,\ldots ,n$.
\end{lemma}
\begin{proof}Suppose to the contrary in \e{f2} we have $x_i=y_i$, $i<k$, and $x_k\neq y_k$. Hence we get
\md1\label{f4}
x_k+\ldots +x_n=y_k+\ldots +y_n.
\emd
From \e{f1} and the relation
\md0
\max A_i\le\frac{1}{4}\cdot \d(A_{i-1})\le \frac{1}{2}\max A_{i-1}
\emd
it follows that
\md7\label{f3}
|x_i|,|y_i|\le \max A_i\le \frac{1}{2}\max A_{i-1}\le\ldots \\\le\frac{1}{2^{i-k-1}}\max A_{k+1}\le \frac{\d( A_k)}{2^{i-k+1}}\le \frac{|x_k-y_k|}{2^{i-k+1}}
\emd
for any $i=k+1,k+2,\ldots, n$. Thus, using \e{f4} and \e{f3}, we get
\md6
|x_k-y_k|\le |x_{k+1}|+|y_{k+1}|+\ldots+|x_n|+|y_n|\\
< 2|x_k-y_k|\sum_{i=1}^\infty\frac{1}{2^{i+1}}=|x_k-y_k|
\emd
which is a contradiction and so  $x_i=y_i$ for all $i=1,2,\ldots ,n$.
\end{proof}
\begin{lemma}\label{L6}
Let $\mu_n$ be a sequence of measures, satisfying the condition \e{1}. Then for any numbers $\Delta>0$ and $0<\delta<1$ there exists a measurable set $A\subset \ZT$, $|A|>0$, such that
\md1\label{d1}
|\{x\in \ZT:\, \sup_{n\in \ZN}S_{\mu_n}\ZI_A(x)>\delta\}|>\Delta \cdot |A|.
\emd
\end{lemma}
\begin{proof}
It is easy to observe that can be supposed each $\supp \mu_n$ is a finite set and moreover
\md0
\mu_n=\sum_{i=l(n-1)+1}^{l(n)} m_i\delta_{x_i},\quad n=1,2,\ldots,
\emd
where $0=l(0)<l(1)<l(2)<\ldots$ are integers, $1>x_i\searrow 0$ and $m_i>0$, $i=1,2,\ldots $. Applying \lem{L4} with $\varepsilon=(1-\delta)/3$  we define finite sets  $E_n$ and $G_n$ with
\md3
&E_n,\,G_n\subset (-x_{l(n)},x_{l(n)}),\quad E_n\cap G_n=\varnothing,\label{d2}\\
&\#(E_n)> \frac{(1-\delta)\#(G_n)}{12},\label{d3}\\
&S_{\mu_n}\ZI_{G_n}(x)>\delta,\quad x\in E_n.\label{d4}
\emd
Clearly we can chose a sequence of integers $n_k$, $k=1,2,\ldots $, satisfying
\md1\label{c4}
\max(E_{n_{k+1}}\cap G_{n_{k+1}})< \frac{\d(E_{n_k}\cap G_{n_k})}{4},\quad k=1,2,\ldots.
\emd
So the sequence of sets $A_k=E_{n_k}\cup G_{n_k}$ satisfies the condition \e{f1}.
Fix an integer
\md1\label{c4}
m>\frac{12\Delta}{1-\delta},
\emd
and denote
\md3
&G=G_{n_1}+G_{n_2}+\ldots+G_{n_m},\label{c5}\\
&F_k=\sum_{i\neq k}G_{n_i}+E_{n_k},\quad E=\cup_{i=1}^nF_i.\label{c6}
\emd
Notice that the sets $F_k$ are mutually disjoint. Indeed, suppose to the contrary $F_p\cap F_q\neq \varnothing$, $p\neq q$, and $x\in F_p\cap F_q$. We then have
\md2
&x=x_1+\ldots+x_m=y_1+\ldots+y_m,\hbox { where }\\
&x_i,y_i\in A_i,\quad x_p\in E_{n_p},\, y_p\in G_{n_p},
\emd
Since $G_{n_p}\cap E_{n_p}=\varnothing$ (see \e{d2}), we have $x_{n_p}\neq y_{n_p}$. On the other hand because $x_i,y_i\in A_i$ and the family $A_i$ satisfies the hypothesis of \lem{L5} we get $x_i=y_i$ for all $i=1,2,\ldots ,m$. This is a contradiction and so $F_k$ are mutually disjoint. Similarly we can prove that any point $x\in G$ has unique representation
\md0
x=x_1+\ldots+x_m,\quad x_i\in G_{n_i},\,i=1,2,\ldots ,m.
\emd
This implies
\md0
\#G=\prod_{i=1}^m\#(G_{n_i}).
\emd
By the same argument, using \e{d3}, we get
\md0
\#F_k=\prod_{i\neq k}\#(G_{n_i})\cdot \#(E_{n_k})\ge \prod_{i\neq k}\#(G_{n_i})\cdot\frac{(1-\delta)\#(G_{n_k})}{12}=\frac{(1-\delta)\# G}{12}.
\emd
Combining this and \e{c4} we conclude
\md1\label{c7}
\#E=\sum_{k=1}^m\#F_k>\frac{m(1-\delta)\#G}{12}>\Delta\cdot \#G.
\emd
To prove \e{d1}, we take an arbitrary $x\in E$. We have $x\in F_k$ for some $1\le k\le m$ and so
\md0
x=x_1+\ldots+x_m,\quad x_i\in G_{n_i},\, i\neq k,\,x_k\in E_{n_k}.
\emd
From \e{c5} it follows that $G_{n_k}\subset G-\sum_{i\neq k}x_i$. Therefore, by \e{d4}, we get
\md0
S_{\mu_{n_k}}\ZI_G(x)=S_{\mu_{n_k}}\ZI_{G-\sum_{i\neq k}x_i}(x_k)\ge S_{\mu_{n_k}}\ZI_{G_{n_k}}(x_k)>\delta.
\emd
Hence we have
\md1\label{c8}
\sup_kS_{\mu_{n_k}}\ZI_G(x)>\delta,\quad x\in E,
\emd
Finally we let $\varepsilon=\d(G\cup E)/2$ and denote
\md0
A=G+(-\varepsilon,\varepsilon),\quad B=E+(-\varepsilon,\varepsilon).
\emd
It is clear that the intervals $t+(-\varepsilon,\varepsilon)$, $t\in G\cup E$, are pairwise disjoint. Hence
\md0
|A|=2\varepsilon\#G,\quad |B|=2\varepsilon\#E,
\emd
and so, by \e{c7} we conclude
\md1\label{c9}
|B|>\Delta |A|.
\emd
Then for an arbitrary $x\in B$ we have $x=t+y$ where $t\in E$ and $|y|<\varepsilon$. Hence, using \e{c8}, we get
\md1\label{c10}
\sup_kS_{\mu_{n_k}}\ZI_A(x)\ge \sup_kS_{\mu_{n_k}}\ZI_{G+y}(x)=\sup_kS_{\mu_{n_k}}\ZI_G(t)>\delta,\quad x\in B.
\emd
Collecting \e{c9} and \e{c10} we obtain \e{d1}.
Lemma is proved.
\end{proof}
\begin{definition}
A sequence of linear operators
\md0
U_n:L^1(\ZT)\to \{\hbox { measurable functions on } \ZT \}.
\emd
is said to be strong sweeping out, if given $\varepsilon >0$ there is a set $E$ with $mE<\varepsilon $ such that $\limsup_{n\to\infty} U_n\ZI_E(x)=1$
and $\liminf_{n\to\infty} U_n\ZI_E(x)=0$ a.e..
\end{definition}
To prove the theorem we need to show that the sequence $S_{\mu_n}$ is strong sweeping out. The following theorem gives a sufficient condition for a sequence of operators to be strong sweeping out.
\begin{theorem}[\cite{Kar}, \S 7, Theorem 6]
 If the sequence of positive translation invariant operators $U_n$ satisfies the conditions
\begin{description}
    \item[a] $U_n(\ZI_\ZT)\to 1$ as $n\to\infty$,
    \item[b] for any $\varepsilon >0$ and $n\in\ZN$ there exists a number $\delta=\delta(\varepsilon,n) >0$,
    such that if $G\subset \ZT$ and $m(G)<\delta$ then
\md1\label{g1}
m\{x\in \ZT:\, U_n\ZI_G(x)>\varepsilon \}<\varepsilon ,
\emd
\item[c] for any $0<\delta<1$ we have
\md0
\sup_{G\subset \ZT,\,|G|>0}\frac{|\{x\in X:\, \sup_{n\in \ZN} U_n\ZI_G(x)\ge\delta \}|}{|G|}=\infty .
\emd
\end{description}
then it is strong sweeping out.
\end{theorem}
Observe, that each $S_{\mu_n}$ is positive translation invariant. The conditions $({\bf a})$ follows from \e{1}. To show $({\bf b})$ we simply note
\md0
\int_\ZT S_{\mu_n}\ZI_G(x)dx=\int_\ZT \int_\ZT\ZI_G(x+t)dtdx=|\mu_n|\cdot |G|,
\emd
and therefore, by Chebishev inequality, we will have \e{g1} provided $|G|<\delta=|\mu_n|/\varepsilon$. The condition $({\bf c})$ immediately follows from \lem{L6}. Theorem is proved
\end{section}

\end{document}